\documentclass[10pt]{amsart}
\usepackage{enumerate,amsmath,amssymb,latexsym,
amsfonts, amsthm, amscd, MnSymbol}

\theoremstyle{plain}

\newtheorem{theorem}{Theorem}

\numberwithin{equation}{section}

\newcommand{\ii}{{\rm i}}

\begin{document}

\title {Elliptic functions from $F(\frac{1}{3}, \frac{2}{3} ; \frac{1}{2} ; \bullet)$ }

\date{}

\author[P.L. Robinson]{P.L. Robinson}

\address{Department of Mathematics \\ University of Florida \\ Gainesville FL 32611  USA }

\email[]{paulr@ufl.edu}

 \keywords{}

\begin{abstract}
Li-Chien Shen developed a family of elliptic functions from the hypergeometric function $_2F_1(\frac{1}{3}, \frac{2}{3} ; \frac{1}{2} ; \bullet)$. We comment on this development, offering some new proofs. 
\end{abstract}

\maketitle

\medbreak

\medbreak 

Shen [1] has presented an interesting construction of elliptic functions based on the hypergeometric function $F(\frac{1}{3}, \frac{2}{3} ; \frac{1}{2} ; \bullet)$. We open our commentary with a brief review of his construction, making some minor notational changes (most of which amount to the dropping of suffixes). 

\medbreak 

Fix $0 < k < 1$ and write 
$$u = \int_0^{\sin \phi} F(\frac{1}{3}, \frac{2}{3} ; \frac{1}{2} ; k^2 t^2) \, \frac{{\rm d} t}{\sqrt{1 - t^2}}$$
so that 
$$\frac{{\rm d} u}{{\rm d} \phi} = F(\frac{1}{3}, \frac{2}{3} ; \frac{1}{2} ; k^2 \sin^2 \phi).$$ 
\medbreak 
\noindent
In a (connected) neighbourhood of the origin, the relation $\phi \mapsto u$ inverts to $u \mapsto \phi$ fixing $0$. Define functions $s, c, d$ by 
$$s(u) = \sin \phi(u)$$
$$c(u) = \cos \phi(u)$$
and
$$d(u) = \phi\,'(u) = 1/F(\frac{1}{3}, \frac{2}{3} ; \frac{1}{2} ; k^2 s^2 (u)).$$ 
Plainly, $s$ and $c$ satisfy the Pythagorean relation  
$$s^2 + c^2 = 1.$$
Shen uses the hypergeometric identity  
$$F(\frac{1}{3}, \frac{2}{3} ; \frac{1}{2} ; \sin^2 z) = \frac{\cos \tfrac{1}{3} z}{\cos z}$$
and the trigonometric triplication formula 
$$4 \cos^3 \tfrac{1}{3} z - 3 \cos \tfrac{1}{3} z = \cos z$$ 
to show that $s$ and $d$ satisfy the relation  
$$d^3 + 3 d^2 = 4(1 - k^2 s^2)$$
or equivalently 
$$4 k^2 s^2 = (1 - d) (2 + d)^2.$$ 
By differentiation, 
$$s\,' = c \, d$$
$$c\,' = - s \,d$$
and 
$$d\,' = - \frac{8}{3} k^2 \frac{s \, c}{2 + d}.$$ 
By means of the $(s, d)$ and $(s, c)$ relations, it follows that $d$ satisfies the differential equation 
$$(d\,')^2 = \frac{4}{9} (1 - d) (d^3 + 3 d^2 + 4 k^2 - 4).$$

\medbreak 

\begin{theorem} \label{d}
The function $d$ is $1 - \frac{4}{9} k^2 (\wp + \tfrac{1}{3})^{-1}$ where $\wp$ is the Weierstrass function with invariants 
$$g_2 = \frac{4}{27} (9 - 8 k^2)$$
and 
$$g_3 = \frac{8}{27^2} (8 k^4 - 36 k^2 + 27).$$ 
\end{theorem} 

\begin{proof} 
Of course, we mean that $d = \phi'$ extends to the stated rational function of $\wp$. In [1] this is proved by appealing to a standard formula for the integral of $f^{-1/2}$ when $f$ is a quartic. Instead, we may work with the differential equation itself, as follows. First, the form of the differential equation 
$$(d\,')^2 = \frac{4}{9} (1 - d) (d^3 + 3 d^2 + 4 k^2 - 4)$$
suggests the substitution $r = (1 - d)^{-1}$: this has the effect of removing the explicit linear factor, thus 
$$(r\,')^2 = \frac{4 k^2}{9} \Big(4 r^3 - \frac{9}{k^2} r^2 + \frac{6}{k^2} r - \frac{1}{k^2}\Big).$$ 
Next, the rescaling $q = \frac{4 k^2}{9} r$ leads to 
$$(q\,')^2 = 4 q^3 - 4 q^2 + \frac{2^5}{3^3} k^2 q - \frac{2^6}{3^6} k^4$$
and the shift $p = q - \frac{1}{3}$ removes the quadratic term on the right side, yielding 
$$(p\,')^2 = 4 p^3 - g_2 p - g_3$$
with $g_2$ and $g_3$ as stated in the theorem. Finally, the initial condition $d(0) = 1$ gives $p$ a pole at $0$; thus $p$ is the Weierstrass function $\wp$ and so $d$ is as claimed. 
\end{proof} 

\medbreak 

 We remark that $\wp$ has discriminant 
$$g_2^3 - 27 g_3^2 = \frac{16^3}{27^3} k^6 (1 - k^2).$$

\bigbreak 

Now, recall that $0 < k < 1$. It follows that the Weierstrass function $\wp$ has real invariants and positive discriminant. Consequently, the period lattice of $\wp$ is rectangular; let $2 K$ and $2 \ii K'$ be fundamental periods, with $K > 0$ and $K' > 0$. We may take the period parallelogram to have vertices $0, 2 K, 2 K + 2 \ii K', 2 \ii K'$ (in counter-clockwise order); alternatively, we may take it to have vertices $\pm K \pm \ii K'$ (with all four choices of sign). The values of $\wp$ around the rectangle $0 \to K \to K + \ii K' \to \ii K' \to 0$ strictly decrease from $+ \infty$ to $- \infty$; moreover, the extreme `midpoint values' satisfy $\wp(K) > 0 > \wp(\ii K')$. In particular, $\wp$ is strictly negative along the purely imaginary interval $(0, \ii K')$. 

\medbreak 

As the Weierstrass function $\wp$ is elliptic of order two, with $2 K$ and $2 \ii K'$ as fundamental periods, the same is true of the function 
$$d = 1 - \tfrac{4}{9} k^2 (\wp + \tfrac{1}{3})^{-1}.$$

\medbreak 

One of the most substantial efforts undertaken in [1] is the task of locating the poles of $d$. As a preliminary step, in [1] Lemma 3.1 it is shown (using conformal mapping theory) that $d$ has a pole in the interval $(0, \ii K')$; as $d$ is even, it also has a pole in $(- \ii K', 0)$. The precise location of the poles of $d$ is announced in [1] Lemma 3.2; the proof of this Lemma is prepared in [1] Section 4 and takes up essentially the whole of [1] Section 5. The approach taken in [1] rests heavily on the theory of theta functions and does more than just locate the poles of $d$. Our approach to locating the poles of $d$ will be more direct: we work with $\wp$ alone, without the need for theta functions. The following is our version of [1] Lemma 3.2. 

\medbreak 

\begin{theorem} \label{poles} 
The elliptic function $d$ has a pole at $\frac{2}{3} \ii K'$. 
\end{theorem} 

\begin{proof} 
Theorem \ref{d} makes it plain that $d$ has a pole precisely where $\wp = - 1/3$. For convenience, write $a = \frac{2}{3} \ii K'$; our task is to establish that $\wp(a) = -1/3$. Recall the Weierstrassian duplication formula 
$$\wp(2 a) + 2 \, \wp(a) = \frac{1}{4} \Big\{ \frac{\wp''(a)}{\wp'(a)}\Big\}^2$$
where 
$$\wp''(a) = 6 \, \wp(a)^2 - \frac{1}{2} \, g_2$$ 
and 
$$\wp'(a)^2 = 4 \, \wp(a)^2 - g_2 \, \wp(a) - g_3.$$ 
Here, 
$$2 a = \frac{4}{3} \ii K' \equiv - \frac{2}{3} \ii K' = - a$$
where the middle congruence is modulo the period $2 \ii K'$ of $\wp$; consequently, 
$$\wp(2 a) = \wp (- a) = \wp (a)$$ 
because $\wp$ is even. The left side of the duplication formula thus reduces to $3 \, \wp(a)$ and we deduce that $b = \wp(a)$ satisfies 
$$3 b = \frac{1}{4} \frac{(6 b^2 - \frac{1}{2} g_2)^2}{4 b^3 - g_2 b - g_3}.$$
Otherwise said, $\wp(a)$ is a zero of the quartic $f$ defined by 
$$f(z) = 12 z (4 z^3 - g_2 z - g_3) - (6 z^2 - \tfrac{1}{2} g_2)^2.$$
For convenience we work with   
$$\frac{27}{4} f(\frac{w}{3}) = w^4 - \frac{2}{3} (9 - 8 k^2) w^2 - \frac{8}{27} (8 k^4 - 36 k^2 + 27) w - \frac{1}{27} (9 - 8 k^2)^2$$
which factorizes as 
$$\frac{27}{4} f(\frac{w}{3}) = (w + 1) \Big(w^3 - w^2 + \frac{1}{3} (16 k^2 - 15) w - \frac{1}{27} (9 - 8 k^2)^2\Big).$$ 
Here, the cubic factor has discriminant 
$$- \frac{4096}{27} k^4 (1 - k^2)^2 < 0$$
and so has just one real zero, which is clearly positive. Thus, $f$ has four zeros: a conjugate pair of non-real zeros, a positive zero and $-1/3$. 
As the values of $\wp$ along $(0, \ii K')$ are strictly negative, it follows that $\wp(a) = - 1/3$ as claimed. 

\end{proof} 

\medbreak 

As an even function, $d$ also has a pole at $- \frac{2}{3} \ii K'$. Both of the poles $\pm \frac{2}{3} \ii K'$ lie in the period parallelogram with vertices $\pm K \pm \ii K'$ and $d$ has order two, so each pole is simple and the accounting of poles (modulo periods) is complete; of course, this may be verified otherwise. 

\medbreak 

We close our commentary with a couple of remarks. 

\medbreak 

In [1] it is mentioned that the squares $s^2$ and $c^2$ are elliptic: as $d$ is elliptic, these facts follow in turn from the $(s, d)$ relation $4 k^2 s^2 = (1 - d) (2 + d)^2$ and the $(s, c)$ relation $c^2 = 1 - s^2$. As $d$ has simple poles, it follows that the poles of $s^2$ and $c^2$ are triple; thus, $s$ and $c$ themselves are not elliptic, as is also mentioned in [1]. Beyond this, the product $s \, c$ is elliptic with triple poles: indeed, 
$$s\, c = - \frac{3}{8 k^2} (2 + d) d\,' = - \frac{3}{16 k^2} \{ (d + 2)^2\}'.$$

\medbreak 

\section*{}

\bigbreak

\begin{center} 
{\small R}{\footnotesize EFERENCES}
\end{center} 
\medbreak 

[1] Li-Chien Shen, {\it On the theory of elliptic functions based on $_2F_1(\frac{1}{3}, \frac{2}{3} ; \frac{1}{2} ; z)$}, Transactions of the American Mathematical Society {\bf 357}  (2004) 2043-2058. 

\medbreak

\end{document}